\documentclass[12pt]{amsart}
\usepackage{amssymb,amsthm,mathrsfs,bm,cite}
\usepackage[centertags]{amsmath}

\newtheorem{thm}{Theorem}[section]
\newtheorem{theorem}[thm]{Theorem}
\newtheorem{corollary}[thm]{Corollary}
\newtheorem{lemma}[thm]{Lemma}
\newtheorem{proposition}[thm]{Proposition}

\theoremstyle{definition}
\newtheorem{definition}[thm]{Definition}

\theoremstyle{remark}
\newtheorem{remark}[thm]{Remark}

\newenvironment{theorem*}[1]{\smallskip\noindent{\bf #1.}\it}{\medskip}

\numberwithin{equation}{section} 
\newcommand\tr{\operatorname{tr}}
\newcommand\dom{\operatorname{dom}}
\newcommand\Iso{\operatorname{Iso}}
\newcommand\sd{\operatorname{\mathbf{sd}}}
\newcommand\SD{\operatorname{SD}}
\newcommand\myRe{\operatorname{Re}}

\newcommand\bC{{\mathbb C}}
\newcommand\bN{{\mathbb N}}
\newcommand\bR{{\mathbb R}}
\newcommand\bZ{{\mathbb Z}}

\newcommand\cD{{\mathcal D}}

\newcommand\cP{{\mathcal P}}
\newcommand\cQ{{\mathcal Q}}

\newcommand\sD{{\mathscr D}}
\newcommand\sI{{\mathscr I}}
\newcommand\sK{{\mathscr K}}
\newcommand\sR{{\mathscr R}}
\newcommand\sX{{\mathscr X}}

\newcommand\al{\alpha}
\newcommand\la{\lambda}

\newcommand\bg{\mathbf{g}}

\newcommand\bu{\mathbf{u}}
\newcommand\bv{\mathbf{v}}

\begin{document}

\title{Inverse spectral problems for energy-dependent Sturm--Liouville equations}

\author[R.~Hryniv and N.~Pronska]{Rostyslav Hryniv and Nataliya Pronska}%

\address[R.H.]{Institute for Applied Problems of Mechanics and Mathematics,
3b~Naukova st., 79601 Lviv, Ukraine \and Institute of Mathematics\\ the University of Rzesz\'{o}w\\ 16\,A Rejtana al.\\ 35-959 Rzesz\'{o}w, Poland}
\email{rhryniv@iapmm.lviv.ua}

\address[N.P.]{Institute for Applied Problems of Mechanics and Mathematics,
3b~Naukova st., 79601 Lviv, Ukraine} \email{nataliya.pronska@gmail.com}



\date{\today}
%\dedicatory{}%
%\commby{}%

\begin{abstract}
We study the inverse spectral problem of reconstructing
energy-dependent Sturm--Liouville equations from their Dirichlet
spectra and sequences of the norming constants. For the class of problems under consideration, we give a complete description of the corresponding spectral data, suggest a reconstruction algorithm, and establish uniqueness of reconstruction. The approach is based on connection between spectral problems for energy-dependent Sturm--Liouville equations and for Dirac operators of special form.
% and essentially uses the well-developed inverse spectral theory for the latter.
\end{abstract}

\maketitle

%%%%%%%%%%%%%%%%%%%%%%%%%%%%%%%%%%%%%%%%%%%%%%%%%

\section{Introduction}

The main aim of the paper is to study the inverse spectral problem
of reconstructing Sturm--Liouville differential equations on~$(0,1)$ with
energy-dependent potentials from their Dirichlet spectra and
suitably defined norming constants. The spectral problem of
interest is given by the differential equation
\begin{equation}\label{eq:intr.spr}
    -y''+qy+2\lambda p y=\lambda^2y
\end{equation}
and the Dirichlet boundary conditions
\begin{equation}\label{eq:intr.bc}
    y(0)=y(1)=0.
\end{equation}
Here~$p$ is a real-valued function in~$L_2(0,1)$ and~$q$ is a
real-valued distribution in $W_2^{-1}(0,1)$; see detailed
definitions in the next section.

Sturm--Liouville spectral problems with potentials depending on
the spectral parameter arise in various models of quantum and
classical mechanics. For instance, to this form can be reduced the corresponding evolution
equations (such as the Klein--Gordon equation~\cite{Naj:83,Jon:93}) that are used to model interactions between colliding
relativistic spinless particles. Then
$\lambda^2$ is related to the energy of the system, thus
explaining the term ``energy-dependent'' in the title of the
paper. Another typical example is related to vibrations of
mechanical systems in viscous media, see~\cite{Yam:90}.

Problems of the form~\eqref{eq:intr.spr} have also appeared
in the physical literature in the context of scattering of waves
and particles. In particular, M.~Jaulent and C.~Jean in
\cite{Jau72, JauJea72, JauJea761,JauJea762}
studied the inverse scattering problems for energy-dependent
Schr\"odinger operators on the line; see also the papers
\cite{MeePiv01,SatSzm95, AktMee91, Kam081,
MakGus89,MakGus86, Nab06,NabGus06, Tsu81}. An
interesting approach to the spectral analysis of the Klein--Gordon equations using the Krein
spaces (i.e., spaces with indefinite scalar products) was
suggested by P.~Jonas~\cite{Jon:93} and H.~Langer, B.~Najman, and C.~Tretter~\cite{LanNajTre:06,Naj:83,LanNajTre:08}.

Non-linear dependence of equation~\eqref{eq:intr.spr} on the
spectral parameter~$\lambda$ suggests
that~\eqref{eq:intr.spr}--\eqref{eq:intr.bc} should be regarded as
a spectral problem for a quadratic operator pencil. Although some
spectral properties of such a problem can easily be derived from
the general spectral theory of polynomial operator pencils~\cite{Mar:88}, there have been rather few papers investigating the inverse problem of
reconstructing the potentials $p$ and $q$ from the suitably
defined spectral data. The problem with $p\in W_2^1(0,1)$ and
$q\in L_2(0,1)$ and with Robin boundary conditions was discussed
by M.~Gasymov and G.~Guseinov in their short paper~\cite{GasGus81}
of 1981 containing no proofs. Such problems for (quasi)-periodic
boundary conditions were considered in, e.g.,
\cite{Gus86,GusNab07,Nab04, Nab07, YanGuo11};
however, typically only Borg-type uniqueness results were
established therein. Some non-classical settings of the inverse
spectral problem (e.g., those with mixed given data,
Hochstadt-type problems, or inverse nodal problems) were discussed
in~\cite{Koy11,Koy06,KoyPan07}.

The main aim of the present paper is to investigate in detail the
inverse spectral problem for
equations~\eqref{eq:intr.spr}--\eqref{eq:intr.bc}, under minimal
smoothness assumptions on real-valued potentials~$p$ and $q$. In
particular, the distributional potential $q$ can include e.\,g.\ the
Dirac delta-functions of Coulomb-like singularities that are widely
used in quantum mechanics to model interactions in molecules and
atoms, see the monographs by S.~Albeverio et al.~\cite{AGHH} and by S.~Albeverio and P.~Kurasov~\cite{AlbKur:99} and the extensive references lists therein. For this wide class of problems, we shall
give a complete description of the corresponding spectral data,
suggest a reconstruction algorithm, and establish uniqueness of
reconstruction. Our approach consists in reducing the spectral
problem  for~\eqref{eq:intr.spr}--\eqref{eq:intr.bc} to the one for
a related Dirac operator; we then study possibility of
reconstructing a Dirac operator of special form from its spectral
data. Since various steps of our considerations are explicit, they
form basis for a reconstruction algorithm.

The paper is organized as follows. In the next section, we introduce the main objects of study and formulate the main results. In Section~\ref{sec:Dir}, we show that the spectral problem~\eqref{eq:intr.spr}--\eqref{eq:intr.bc} is closely related to the one for a Dirac operator of a special form. Further we discuss the possibility to transform Dirac operators to some canonical form by means of the so-called transformation operators. Existence and uniqueness of a special Dirac operator (and thus of the corresponding operator pencil) possessing the prescribed spectral data is tackled with in Sections~\ref{sec:Dir-rec} and \ref{sec:Dir-rec-uniq} respectively. Finally, the last Section~\ref{sec:alg} summarizes the reconstruction algorithm and  discusses some possible extensions.

\emph{Notations.} Throughout the paper, $L_{2,\mathbb{R}}(0,1)$ and $W_{2,\mathbb{R}}^{-1}(0,1)$ will stand for the sets of real-valued functions in~$L_2(0,1)$ and distributions in $W_{2}^{-1}(0,1)$, respectively. We denote by $\rho(T)$ and $\sigma(T)$
the resolvent set and the spectrum of a linear operator or a quadratic operator pencil~$T$, and by $\mathcal{M}_2=\mathcal{M}_2(\bC)$ the linear space of $2\times 2$ matrices with complex entries endowed with the Euclidean operator norm. The superscript~$\mathrm{t}$ designates transposition of vectors and matrices, e.g., $(c_1,c_2)^{\mathrm{t}}$ is the column vector~$\binom{c_1}{c_2}$.

%%%%%%%%%%%%%%%%%%%%%%%%%%%%%%%%%%%%%%%%%%%%%%%%%%%%%

\section{Preliminaries and main results}

%%%%%%%%%%%%%%%%%%%%%%%%%%%%%%%%%%%%%%%%%%%%%%%%%%%%%

Equation~\eqref{eq:intr.spr} contains terms depending both on $\lambda$ and $\lambda^2$ and therefore leads to a spectral problem for a quadratic operator pencil that we now introduce.

To begin with, we  recall that the potential~$q$ in~\eqref{eq:intr.spr} is a real-valued distribution in the Sobolev space~$W_2^{-1}(0,1)$ and thus $q=r'$ for some~$r\in L_{2,\mathbb{R}}(0,1)$. The Sturm--Liouville operator with potential~$q$ can be defined following the regularization method due to Savchuk and Shkalikov~\cite{SavShk:1999,SavShk:2003}. Namely, for every absolutely continuous function~$y$, we denote by $y^{[1]}:=y'-ry$ its \emph{quasi-derivative} and introduce the differential expression
\[
    \ell (y):= -\bigl(y^{[1]}\bigr)' - r y^{[1]} - r^2 y
\]
acting on
\[
    \dom \ell = \{y \in AC (0,1) \mid y^{[1]} \in AC[0,1], \ \ell(y) \in L_2(0,1)\}.
\]
We now define the operator~$A$ via
\[
    Ay = \ell (y)
\]
on the domain
\[
    \dom A:=\{y \in\dom \ell \mid y(0)=y(1) =0\}.
\]

A straightforward verification shows that $\ell (y) = - y'' + qy$ in
the sense of distributions and, if $q$ is integrable, even in the usual sense.
Therefore for regular~$q$, $A$ is the standard Sturm--Liouville operator with potential~$q$ and Dirichlet boundary conditions. It is
known~\cite{SavShk:1999,SavShk:2003} that if $q\in W_2^{-1}(0,1)$
is real-valued, then the operator $A$ is self-adjoint, bounded
below, and has a simple discrete spectrum. The Green function of the operator~$A$ is continuous in the square $[0,1]\times[0,1]$, so that the resolvent of~$A$ is of Hilbert--Schmidt class.

Further, we denote by $B$ the operator of multiplication by the
potential~$p\in L_2(0,1)$. The operator~$B$ is in general
unbounded; however, since $BA^{-1}$ is of Hilbert--Schmidt class,
$B$ is $A$-compact~\cite[Ch.~IV]{Kat:1966}. In particular, $\dom B
\supset\dom A$ and $B$ is bounded relative to~$A$ with relative
$A$-bound~$0$~\cite[Lemma~III.2.16]{EngNag:2000}. Finally, $I$
stands for the identity operator in~$L_2(0,1)$.

Now the spectral problem~\eqref{eq:intr.spr}--\eqref{eq:intr.bc} can
be regarded as the spectral problem for the \emph{quadratic
operator pencil} $T_{p,q}$ defined as
\[
   T_{p,q}(\la):=\la^2I - 2 \la B - A
\]
for $\lambda\in \bC$ on the $\la$-independent domain~$\dom T_{p,q}:=
\dom A$. The above-mentioned properties of the operators $A$ and $B$
imply that, for every $\lambda\in\bC$, the
operator~$T_{p,q}(\lambda)$ is well defined and closed on $\dom
T_{p,q}$. Before continuing, we recall the following notions of the spectral theory
of operator pencils, see~\cite{Mar:88}.

\begin{definition}
The \emph{spectrum}~$\sigma(T_{p,q})$ of the operator
pencil~$T_{p,q}$ is the set of all $\lambda\in\bC$ for
which~$T_{p,q}(\lambda)$ is not invertible, i.e.,
\[
    \sigma(T_{p,q})=\{\lambda\in\mathbb{C}\mid 0\in\sigma(T_{p,q}(\lambda))\}.
\]
A number $\lambda\in\bC$ is called the \emph{eigenvalue} of
$T_{p,q}$ if $T_{p,q}(\lambda)y=0$ for some non-zero
function~$y\in\dom T_{p,q}$, which is then the corresponding
\emph{eigenfunction}. Finally,
\[
    \rho(T_{p,q}):=\mathbb{C}\setminus\sigma(T_{p,q})
\]
is the \emph{resolvent set}~ of the operator pencil~$T_{p,q}$.
\end{definition}

It was shown in~\cite{Pro:2011a} that the spectrum of the
pencil~$T_{p,q}$ consists entirely of eigenvalues and that
$\sigma(T_{p,q})$ is a discrete subset of~$\bC$.
In general, $T_{p,q}$ can possess non-real and/or non-simple
eigenvalues; the latter means that the algebraic multiplicity of
an eigenvalue can be greater than~$1$~\cite{Mar:88}. The
approach to the inverse spectral problem for $T_{p,q}$ we are
going to use is currently well understood only for the case where
the spectrum of~$T_{p,q}$ is real and simple. This happens e.\,g.\ when there is a
$\la_0\in\bR$ such that the operator~$T_{p,q}(\la_0)$ is negative~\cite{Pro:2011a}, i.e., when $T_{p,q}$ is a hyperbolic pencil~\cite[Ch.~31]{Mar:88}.
Making the shift $\la \mapsto \la +\la_0$ if necessary, we may
(and shall) assume that $\la_0=0$.

Therefore, our standing assumption throughout the paper is that
\begin{itemize}
\item[(A)] $p\in L_{2,\mathbb{R}}(0,1)$, $q\in W_{2,\mathbb{R}}^{-1}(0,1)$, and the operator~$A$ is positive.
\end{itemize}

Under assumption~(A), the eigenvalues of~$T_{p,q}$ are all real,
simple, and can be labeled in increasing order as $\la_n$ for
$n\in\bZ^*:=\bZ\setminus\{0\}$ so that
\[
    \la_n = \pi n + p_0 + \tilde \la_n,
\]
where $p_0:=\int_0^1 p(x)\,dx$ and $(\tilde\la_n)$ is a sequence
in~$\ell_2(\bZ^*)$~\cite{Pro:2011a,Pro:2012}. Observe also that
under assumption~(A) the point $\la=0$ is in the resolvent set
of~$T_{p,q}$.

We observe that equation~\eqref{eq:intr.spr} can be recast using quasi-derivatives as
\[
    \ell(y) +2\la p y = \la^2y
\]
and that for every complex $a$ and $b$ it possesses a unique solution satisfying the initial conditions $y(0)=a$ and $y^{[1]}(0)=b$. This allows to introduce the norming constants in the following way.

\begin{definition}
For an eigenvalue $\la_n$ of $T_{p,q}$, denote by $y_n$ the
corresponding eigenfunction normalized by the initial
conditions~$y_n(0)=0$ and $y_n^{[1]}(0)=\la_n$. Then the quantity
\begin{equation}\label{eq:pre.al}
    \alpha_n:=2\int_0^1y_n^2(t)dt-\frac{2}{\lambda_n}\int_0^1p(t)y^2_n(t)dt
\end{equation}
is called the \emph{norming constant} corresponding to the
eigenvalue~$\lambda_n$.
\end{definition}

Although this definition looks somewhat artificial, there is a good
reason for defining the norming constants via~\eqref{eq:pre.al}.
Firstly, for $p\equiv0$, problem
\eqref{eq:intr.spr}--\eqref{eq:intr.bc} becomes the spectral
problem for the Sturm--Liouville operator~$A$, and \eqref{eq:pre.al} up
to a constant factor agrees with the common definition of the
norming constant~\cite{GelLev:51}. Secondly, denoting by $T'_{p,q}$ the
$\lambda$-derivative of~$T_{p,q}$, we find that
\[
    \bigl(T'_{p,q}(\la_n)y_n,y_n\bigr) =  \la_n \al_n,
\]
whence $\al_n$ determines the \emph{type} of the
eigenvalue~$\la_n$, see~\cite{Mar:88}. In the next section, we
shall transform the spectral problem for the operator pencil to
that for some  Dirac operator~$\sD(P)$, and \eqref{eq:pre.al}
agrees with the standard definition of the norming constant
for~$\sD(P)$. %~\cite{GasLev:1966}.
Next, we mention that the function
$u(x,t):=y_n(x)e^{\la_n t}$ is a solution of the corresponding
evolution equation $T_{p,q}(d/dt)u=0$, and its \emph{energy}
$(Au,u)+ (\dot{u},\dot{u})$, with $\dot{u}$ denoting $\partial
u/\partial t$, is equal to $\la_n^2
 \al_n e^{2\la_n t}$.

\begin{remark}\label{rem:pre.alpha} It follows from~\cite{Pro:2011a}
that if $A$ is invertible, then its positivity (cf.\ assumption~(A)) is necessary and
sufficient for all $\al_n$ to be positive.
\end{remark}

\begin{definition} Assume that $\la$ is an eigenvalue of the quadratic operator pencil~$T_{p,q}$ and that $\al$ is the corresponding norming constant. Then~$(\la,\al)$ is called the \emph{spectral eigenpair} of~$T_{p,q}$. The \emph{spectral data} $\sd(T_{p,q})$ of the pencil~$T_{p,q}$ is the set
\[
    \sd(T_{p,q}):=\{(\la,\al) \mid \la \in \sigma(T_{p,q})\}
\]
of all its spectral eigenpairs.
\end{definition}

The inverse spectral problem of interest is to reconstruct the
potentials~$p$ and $q$ of the operator pencil~$T_{p,q}$ given its
spectral data~$\sd(T_{p,q})$. Some properties of spectral data for the class of quadratic pencils~$T_{p,q}$ under consideration was established in~\cite{Pro:2011a}. Our aim in this paper is, firstly, to give a complete description of the set of the spectral data and, secondly, to find and justify an algorithm reconstructing the potentials $p$ and $q$ from the spectral data.

We observe that when $p\equiv 0$, the spectral problem for the
operator pencil~$T_{0,q}$ becomes the spectral problem $Ay =
\la^2y$ for the Sturm--Liouville operator~$A$.
Then $\la_{-n} = - \la_n$, $\al_{-n}=\al_n$, and $\al_n$ of~\eqref{eq:pre.al} agrees with the
standard definition of a norming constant~\cite{GelLev:51}.
For the Sturm--Liouville operator~$A$ with a real-valued $q\in L_1(0,1)$ and Robin boundary conditions, it was proved in~\cite{GelLev:51}
that the spectrum~$(\la_n^2)_{n\in\bN}$ of~$A$ and the
sequence $(\al_n)_{n\in\bN}$ of the corresponding norming constants
uniquely determine the potential~$q$; the case of a
distributional potential~$q\in W_2^{-1}(0,1)$ was treated in e.g.~\cite{HryMyk:2003a,ColMcL:93,SavShk:2005}.
The above references also suggest algorithms of reconstructing the potential~$q$ from the spectral data of~$A$.

The operator pencil~$T_{p,q}$ contains two real-valued potentials
$p$ and $q$ to be determined in the inverse problem; however, the
spectral data of~$T_{p,q}$ are twice as large as for a standard
Sturm--Liouville operator. Therefore one may hope that the inverse
spectral problem of reconstructing $p$ and $q$ from the spectral
data of~$T_{p,q}$ is well posed. Our first main result gives uniqueness of reconstruction.

\begin{theorem}%[Uniqueness]
\label{thm:pre.uniq}
Under assumption~(A), the operator pencil~$T_{p,q}$ is uniquely
determined by its spectral data~$\sd(T_{p,q})$.
%spectrum~$\bla$ and the sequence of norming constants~$\bal$.
\end{theorem}

It follows from~\cite{Pro:2011a,Pro:2012} (cf. also the result of~\cite{GasGus81} for $p\in W_2^1(0,1)$ and $q\in L_2(0,1)$) that the
spectral data for the operator pencils under consideration belong
to the following set.

\begin{definition}\label{def:SD}
We denote by~$\SD$ the family of all sets~$\{(\la_n,\al_n)\}_{n\in\bZ^*}$ consisting of pairs~$(\la_n,\al_n)$ of real numbers satisfying the following properties:
\begin{enumerate}
    \item[(i)] $\lambda_n$ are nonzero, strictly increase with $n\in\bZ^*$, and have the representation $\lambda_n = \pi n + h + \tilde{\lambda}_n$ for some $h\in\bR$ and a sequence $(\tilde\la_n)$ in $\ell_2(\bZ^*)$;
    \item[(ii)] $\alpha_n>0$ for all $n\in\bZ^*$ and the numbers $\tilde \alpha_n:=\alpha_n-1$ form an $\ell_2(\bZ^*)$-sequence.
\end{enumerate}
\end{definition}

Our second result claims that, conversely, every element of~$\SD$ is spectral data for some operator pencils~$T_{p,q}$ under consideration. In particular, it shows that conditions~(i)--(ii) above give a complete description of the spectral data for the operator pencils~$T_{p,q}$ with $p\in L_{2,\mathbb{R}}(0,1)$ and~$q\in W_{2,\mathbb{R}}^{-1}(0,1)$.

\begin{theorem}\label{thm:pre.exist}%[Existence]
For every~$\sd\in\SD$, there exists~$p\in L_{2,\mathbb{R}}(0,1)$ and~$q\in W_{2,\mathbb{R}}^{-1}(0,1)$ such that~$\sd$ is the spectral data for the operator pencil~$T_{p,q}$.
\end{theorem}

The proof of this theorem is constructive and suggests the explicit reconstruction algorithm determining the potentials~$p$ and $q$ from the set~$\sd$ belonging to~$\SD$; see Section~\ref{sec:alg}.

Our approach consists in reducing the spectral problem for $T_{p,q}$
to the one for a Dirac operator of a special form acting in~$L_2(0,1)\times L_2(0,1)$, see Section~\ref{sec:Dir}. Under a suitable unitary
gauge transformation this Dirac operator takes the ``shifted'' AKNS
normal form. For AKNS Dirac operators, the direct and inverse
spectral problems are well understood,
see~\cite{LevSar:1991,GasDza:1966,AlbHryMk:2005:RJMP}. We shall use
the known methods to first reconstruct the Dirac operator in the
``shifted'' AKNS form from the given data and then to transform this
Dirac operator to the one that is directly associated with some
pencil~$T_{p,q}$. The latter gives the required solution of the
inverse spectral problem of interest.

%%%%%%%%%%%%%%%%%%%%%%%%%%%%%%%%%%%%%%%%%%%%%%%%%%

\section{Reduction to the Dirac system}\label{sec:Dir}

%%%%%%%%%%%%%%%%%%%%%%%%%%%%%%%%%%%%%%%%%%%%%%%%%%

In this section we shall show that under the standing assumption~(A) the spectral problem for the operator pencil~$T_{p,q}$ can be reduced to the one for a special Dirac operator. We start with the following observation.

\begin{lemma}\label{lem:Dir.pos-sol}
Under the standing assumption~(A) the equation $\ell(y)=0$ possesses a
solution~$y$ that is strictly positive on~$[0,1]$.
\end{lemma}

\begin{proof}
For $h\in\bR$, consider an operator~$A_h$ defined via $A_hy=\ell (y)$ on the set of functions in~$\dom \ell$ verifying the boundary conditions
\[
    y(0)-hy^{[1]}(0)=y(1)=0.
\]
The operators $A_h$ are self-adjoint and have discrete spectra. We show that~$A_h$ depend continuously on~$h$ in the norm resolvent sense.

Indeed, fix a non-real $\la$ and denote by $\varphi_-$ and $\varphi_+$ solutions of the equation~$\ell(y) = \la y$ satisfying the initial conditions~$\varphi_-(0)=0$, $\varphi_-^{[1]}(0)=1$ and the terminal conditions~$\varphi_+(1)=0$, $\varphi_+^{[1]}(1)=1$, respectively. Take an arbitrary $f\in L_2(0,1)$ and set $g_h:= (A_h - \la)^{-1}f$. The function $z:= g_h-g_0$ solves the equation $\ell(z) = \la z$ and satisfies the terminal condition~$z(1)=0$; therefore there exists a constant $c_h(f)$ such that $z=c_h(f)\varphi_+$.
Since the function $g_h = g_0 + c_h(f)\varphi_+$ satisfies the initial condition~$g_h(0) = h g_h^{[1]}(0)$, we conclude that
\[
    c_h(f) = \frac{h g_0^{[1]}(0)}{\varphi_+(0)-h\varphi_+^{[1]}(0)}.
\]
The denominator never vanishes for $h\in\bR$ since the non-real~$\la$ is in the resolvent set of all operators $A_h$.

With
    $W:= \varphi_+\varphi_-^{[1]}-\varphi_+^{[1]}\varphi_-$
denoting the modified Wronskian, we find that $W$ is constant on $[0,1]$ and that
\[
    g_0(x)
       =  \frac{\varphi_-(x)}{W}\int_x^1\varphi_+(t)f(t)\,dt
        + \frac{\varphi_+(x)}{W}\int_0^x\varphi_-(t)f(t)\,dt,
\]
whence
\[
    g_0^{[1]}(0) = \frac1W \int_0^1\varphi_+(t)f(t)\,dt.
\]
Therefore the resolvent difference $(A_h-\la)^{-1}-(A_0-\la)^{-1}$ is a rank one operator equal to
\[
    \frac{h}{W}\frac{(\cdot, \overline{\varphi_+}) \varphi_+}{{\varphi_+(0)-h\varphi_+^{[1]}(0)}}
\]
and thus the norm resolvent continuity of $A_h$ as a function of~$h\in \bR$ follows.

By assumption~(A), the operator $A=A_0$ is uniformly positive; thus the above continuity in~$h$ implies that there is an $h>0$ such that $A_h \gg 0$. Fix one such an $h$ and denote by $y_h$ the solution of the equation $\ell(y)=0$ satisfying the initial conditions $y(0)=h$ and $y^{[1]}(0)=1$. We claim that $y_h$ stays positive over~$[0,1]$.

To prove this, we consider the quadratic form $\mathfrak{a}_h$ of the operator~$A_h$. Integration by parts shows that
\[
    \mathfrak{a}_h[y] = \|y'\|^2 - 2 \myRe (ry',y) + \frac1h|y(0)|^2
\]
for $y\in\dom A_h$. Next (cf.~\cite{HryMyk:2001}), the quadratic form $\myRe (ry',y)$ is relatively bounded
with respect to the form
\[
    \tilde{\mathfrak{a}}_h[y] := \|y'\|^2 + \frac1h|y(0)|^2
\]
with relative bound~$0$. Since the quadratic form $\tilde{\mathfrak{a}}_h$ is closed on the domain
\[
    \dom \tilde{\mathfrak{a}}_h :=\{y\in W_2^1(0,1) \mid y(1)=0\},
\]
Theorem~VI.1.33 of~\cite{Kat:1966} implies that the quadratic form
$\mathfrak{a}_h$ is closed on the same domain.

Now assume, on the contrary, that the above function~$y_h$ does not remain positive over~$[0,1]$  and denote by $x_0\in(0,1]$ a zero of $y_h$. Then the function
\[
    z(x)
        :=\begin{cases}
            y_h(x) & \text{ for } x<x_0; \\
            0      & \text{ for } x\ge x_0 \\
          \end{cases}
\]
belongs to the domain of the quadratic form $\mathfrak{a}_h$ of~$A_h$. Moreover,
integration by parts in the expression
\[
    \mathfrak{a}_h[z]
        = \int_0^{x_0} |y_h'(x)|^2\,dx
        - 2 \myRe\int_0^{x_0} (r y'_h\overline {y_h})(x)\,dx +\frac1h|y_h(0)|^2
\]
shows that $\mathfrak{a}_h[z]=0$, which contradicts uniform positivity of the operator~$A_h$. The derived contradiction shows the assumption on $y_h$ was wrong, and thus $y_h$
is a solution of $\ell(y)=0$ that is positive over~$[0,1]$. The proof is complete.
\end{proof}

Now we take a solution $y$ of the equation~$\ell(y)=0$ that stays positive on~$[0,1]$ and set~$v:=y'/y$. Since $y' = y^{[1]} + ry \in L_2(0,1)$, the function~$v$ is in $L_2(0,1)$. Moreover, direct calculations show that $q=v'+v^2$ (so that $q$ is a Miura
potential, see~\cite{KapPerShuTop:2005}) and that the operator~$A$ can
be written in the factorized form, viz.
\begin{equation}\label{eq:Dir.Afact}
    Ay  = -\Bigl(\frac{d}{dx}+v\Bigr)\Bigl(\frac{d}{dx}-v\Bigr)y.
\end{equation}
We observe that there are many different $v$ satisfying
the Riccati equation~$v'+v^2 = q$ and thus allowing the above
factorization of~$A$; we fix one such a $v$ in what follows and notice that the function $v-r$ is continuous.

For $\la\ne0$ we can recast the spectral
problem~\eqref{eq:intr.spr} as a first order system for the
functions~$u_2:=y$ and~$u_1:=(y'-vy)/\la$, namely,
\begin{align}\label{eq:Dir.system1}
    u_2' - v u_2 &= \la u_1,\\
    -u_1'-v u_1 + 2 p u_2 &= \la u_2. \label{eq:Dir.system2}
\end{align}
Setting
\begin{equation}\label{eq:Dir.P}
    J:=\left(
           \begin{array}{cc}
             0 & 1 \\
             -1 & 0 \\
           \end{array}
         \right),
\qquad
    P:=\left(
                \begin{array}{cc}
                  0 & -v \\
                  -v & 2p \\
                \end{array}
              \right),
\qquad
    \bu=\binom{u_1}{u_2},
\end{equation}
we see that the above system is the spectral problem
$\sD(P)\bu=\la\bu$ for the Dirac operator $\sD(P)$
acting in $L_2(0,1)\times L_2(0,1)$ via
\begin{equation}
\label{eq:dif.exp.Dirac}
    \sD(P)\bu =J\frac{d\bu}{dx}+P\bu =: \ell(P)\bu
\end{equation}
on the domain
\[
    \dom\sD(P):=\{\bu=(u_1,u_2)^\mathrm{t}\in W_2^1(0,1)\times W_2^1(0,1)
        \mid u_2(0)=u_2(1)=0\}.
\]

\begin{lemma}\label{lem:Dir.spectra}
The nonzero spectra of the Dirac operator~$\sD(P)$ and the
operator pencil~$T_{p,q}$ coincide.
\end{lemma}

\begin{proof}
We start by observing that the Dirac operator~$\sD(P)$ is
self-adjoint in the Hilbert space~$L_2(0,1)\times L_2(0,1)$ and
has a simple discrete spectrum~\cite{LevSar:1991}. The above arguments show that
every eigenvalue $\la$ of~$T_{p,q}$ is an eigenvalue of
$\sD(P)$ as well.

Conversely, assume that $\la\ne0$ is an eigenvalue
of~$\sD(P)$ and $\bu=(u_1,u_2)^\mathrm{t}$ is a
corresponding eigenfunction. Determining $u_1$ via $u_2$
from~\eqref{eq:Dir.system1} and substituting
in~\eqref{eq:Dir.system2}, we find that $y:=u_2$ satisfies the
equality
\[
    -\frac1\la\Bigl(\frac{d}{dx}+v\Bigr)\Bigl(\frac{d}{dx}-v\Bigr)y
        - 2 p y = \la y,
\]
which in view of~\eqref{eq:Dir.Afact} yields~$T_{p,q}y=0$.
Clearly, $y$ is non-trivial and satisfies the Dirichlet boundary
conditions; henceforth it is an eigenfunction of~$T_{p,q}$
corresponding to the eigenvalue~$\la$.
\end{proof}

\begin{remark}\label{rem:Dir.0}
A straightforward analysis
of~\eqref{eq:Dir.system1}--\eqref{eq:Dir.system2} shows that
$\lambda=0$ is an eigenvalue of the Dirac
operator~$\sD(P)$, the corresponding eigenfunction being
$\bu=(u_1,u_2)^\mathrm{t}$ with $u_1 = \exp(-\int v)$ and $u_2\equiv0$.
However, under assumption~(A) the number $\lambda=0$ is never an
eigenvalue of~$T_{p,q}$; therefore,
\[
    \sigma\bigl(\sD(P)\bigr) = \sigma(T_{p,q}) \cup \{0\}.
\]
\end{remark}

The norming constant for the Dirac operator~$\sD(P)$
corresponding to an eigenvalue~$\la$ is defined as
\[
    \|\bu\|^2= \|u_1\|^2_{L_2} + \|u_2\|^2_{L_2},
\]
where~$\bu=(u_1,u_2)^\mathrm{t}$ is the eigenfunction for~$\la$
normalized by the initial conditions $u_1(0)=1$ and $u_2(0)=0$.
Assume that $\la\ne0$; then, as shown in the proof of
Lemma~\ref{lem:Dir.spectra}, $y:=u_2$ is an eigenfunction
of~$T_{p,q}$ corresponding to the eigenvalue~$\la$; note also that
$y$ is real valued and satisfies the initial conditions $y(0)=0$ and $y^{[1]}(0)=(y'-ry)(0)=(y'-vy)(0)=\la u_1(0)=\la$. Integration by parts on account of~\eqref{eq:Dir.Afact} gives
\[
    \|u'_2-vu_2\|^2_{L_2} = (Au_2,u_2)_{L_2};
\]
using now~\eqref{eq:Dir.system1}, we conclude that
\[
    \|\bu\|^2 = \frac1{\la^2}(Ay,y) + (y,y) = -\frac2\la(By,y) + 2(y,y),
\]
which coincides with~\eqref{eq:pre.al}. We have thus established
the following important result.

\begin{lemma}\label{lem:Dir.normconst}
The norming constants corresponding to nonzero eigenvalues of the
Dirac operator~$\sD(P)$ and the operator pencil~$T_{p,q}$
coincide.
\end{lemma}

The above lemma suggests that we can use the spectral data $\sd(T_{p,q})$ of the
operator pencil~$T_{p,q}$ in order to find the related Dirac operator~$\sD(P)$. Having
determined the potential~$P=(p_{ij})_{i,j=1}^2$ of $\sD(P)$, we then
identify the potentials~$p$ and $q$ of the operator pencil~$T_{p,q}$
as $p:=p_{22}/2$ and $q: = -p'_{12}+p_{12}^2$. We note, however,
that since the factorization~\eqref{eq:Dir.Afact} is not unique,
there are many Dirac operators~$\sD(P)$ associated with $T_{p,q}$.
Therefore the spectral data~$\sd(T_{p,q})$ cannot determine such an
operator~$\sD(P)$ uniquely.

The reason for this non-uniqueness is quite clear from
Remark~\ref{rem:Dir.0} and Lemma~\ref{lem:Dir.normconst}; indeed,
the spectral data for $T_{p,q}$ leave the norming constant
$\alpha_0$ for the eigenvalue $\lambda_0:=0$ of $\sD(P)$
undetermined. It is this freedom in the choice of $\alpha_0$ that
leads to non-uniqueness of potentials~$P$ for the associated Dirac
operators~$\sD(P)$. However, we shall show in Section~\ref{sec:Dir-rec-uniq} that all such Dirac operators determine the same pencil~$T_{p,q}$.

%%%%%%%%%%%%%%%%%%%%%%%%%%%%%%%%%%%%%%%%%%%%%%%%%%

\section{Transformation operators}\label{sec:transf}

%%%%%%%%%%%%%%%%%%%%%%%%%%%%%%%%%%%%%%%%%%%%%%%%%%

We shall use the relation between the operator pencil~$T_{p,q}$ and the Dirac operator~$\sD(P)$ of~\eqref{eq:Dir.P}--\eqref{eq:dif.exp.Dirac} explained in the previous section and the well-developed inverse spectral theory for Dirac operators to reconstruct~$\sD(P)$ from the given spectral data. Once such a Dirac operator has been found, it is then straightforward to determine the corresponding potentials~$p$ and~$q$ of the pencil~$T_{p,q}$.

However, the classical inverse spectral theory reconstructs a Dirac operator with potential in the AKNS form or in other canonical form, see Section~\ref{sec:Dir-rec}. Therefore we then have to transform such a canonical Dirac operator to the one of the form~\eqref{eq:Dir.P} keeping the spectral data unchanged. This is done by means of the so-called transformation operators, which we study in this section.

More exactly, assume that $P$ and $Q$ are $2\times 2$ matrix-valued potentials in $L_2((0,1),\mathcal{M}_2)$ and set
\[
    \cD_0:=\{(u_1,u_2)^{\mathrm{t}}\in W_2^1(0,1)\times W_2^1(0,1)
    \mid u_2(0)=0\}.
\]
We need a \emph{transformation operator}  $\sX=\sX(P,Q)$ between
the Dirac operators~$\ell(P)$ and $\ell(Q)$ acting on the set~$\cD_0$,
i.e., for an operator satisfying the
relation~$\sX\ell(P)\bu=\ell(Q)\sX\bu$ for all $\bu\in\cD_0$.
Similar transformation operators for $P$ and $Q$ with Lipschitz
continuous entries were constructed
in~\cite{LevSar:1991,CoxKno:1996}; it is thus reasonable to look for the transformation operator~$\sX$ of a similar form
\begin{equation}
\label{eq:Tr.op} \sX\bu(x)=R(x)\bu(x)+\int_0^xK(x,s)\bu(s)ds,
\end{equation}
where $R$ and $K$ are $2\times 2$ matrix-valued functions of one and two variables respectively. Keeping in mind that the Dirac operators of interest, $\ell(P)$ and $\ell(Q)$, are considered on functions satisfying the same initial conditions, we impose the restriction $R(0)=I$ guaranteeing that $\sX$ preserves the values of functions at $x=0$. Under such a normalization, $R$ will explicitly be given as
\begin{equation}
    \label{eq:Dir.R}
    R(x)=e^{\theta_1(x)}\left(%
    \begin{array}{cc}
    \cos\theta_2(x) & \sin\theta_2(x) \\
     -\sin\theta_2(x) & \cos\theta_2(x) \\
    \end{array}%
    \right) = e^{\theta_1(x)I+\theta_2(x)J},
\end{equation}
with
\begin{equation}\label{eq:Dir.theta}
   \begin{aligned}
        \theta_1(x)=\frac{1}{2}\int_0^x\mathrm{tr}[J(Q(s)-P(s))]\,ds,\\
        \theta_2(x)=\frac{1}{2}\int_0^x\mathrm{tr}(Q(s)-P(s))ds,
    \end{aligned}
\end{equation}
cf.~\cite{CoxKno:1996}. More exactly, the following analogue of
Theorem~3.1 of~\cite{CoxKno:1996} holds true.

\begin{theorem}\label{thm:Dir.tr-op}
Assume that $P$ and $Q$ are in $L_2((0,1),\mathcal{M}_2)$. Then an  operator~$\sX(P,Q)$ of the form~\eqref{eq:Tr.op}, with~$R$ obeying the condition~$R(0)=I$ and with a summable kernel~$K$, is a
transformation operator for~$\ell(P)$ and~$\ell(Q)$ on the set~$\cD_0$ if and only if the matrix-valued function~$R$ is given by~\eqref{eq:Dir.R}--\eqref{eq:Dir.theta} and the
kernel~$K=(K_{ij})_{i,j=1}^2$ is a mild solution of the partial differential equation
\begin{equation}
\label{eq:Dir.K}
    J\partial_xK(x,y)+\partial_yK(x,y)J=K(x,y)P(y)-Q(x)K(x,y)
\end{equation}
in the domain~$\Omega\equiv\{(x,y)\mid 0<y<x<1\}$ satisfying for~$0\leq x\leq1$ the boundary conditions
\begin{align}\label{eq:K_b.c_1}
    K(x,x)J-JK(x,x)&=JR'(x)+Q(x)R(x)-R(x)P(x),\\
    K_{12}(x,0)&=K_{22}(x,0)=0.\label{eq:K_b.c_2}
\end{align}
\end{theorem}

The proof of this theorem follows in general that of Theorem~3.1
of~\cite{CoxKno:1996}. One essential difference is that since the
matrix-valued functions~$P$ and $Q$ are less regular, $K$ need not
be differentiable in the usual sense. Therefore differentiation
of~$K$ should be understood in the distributional sense; that is
why $K$ is only required to be a mild solution
of~\eqref{eq:Dir.K}. For the sake of completeness, we justify the
steps involving differentiation in the proof below.

\begin{proof}
Firstly, we observe that, for an integrable $K$ and for $\bu\in\cD_0$, the standard formula
\[
    \frac{d}{dx}\int_0^x K(x,s)\bu(s)ds
        =K(x,x)\bu(x)+\int_0^x \partial_x K(x,s)\bu(s)ds
\]
remain to hold in the sense of distributions, so that
\begin{equation}\label{eq:XD}
    \begin{aligned}
        \ell(Q)\sX\bu(x)
            &=JR(x)\bu'(x)+\{JR'(x)+Q(x)R(x)\}\bu(x)\\
            &\ +J\frac{d}{dx}\int_0^x K(x,s)\bu(s)ds +
                \int_0^xQ(x)K(x,s)\bu(s)ds\\
            &=JR(x)\bu'(x)+\{JR'(x)+Q(x)R(x)+JK(x,x)\}\bu(x)\\
            &\ +\int_0^x\{J\partial_xK(x,s)+Q(x)K(x,s)\}\bu(s)ds.
    \end{aligned}
\end{equation}
Similarly, the integration by parts formula
\[
    \int_0^x K(x,s)\bu'(s)\,ds
        =K(x,x)\bu(x)-K(x,0)\bu(0) -\int_0^x \partial_sK(x,s)\bu(s)\,ds
\]
holds in the distributional sense, and in the same sense one gets the equalities
\begin{equation}\label{eq:DX}
\begin{aligned}
    \sX\ell(P)\bu(x)
        &= R(x)J\bu'(x) + R(x)P(x)\bu(x)\\
        &\ +\int_0^x K(x,s)\{J\bu'(s)+P(s)\bu(s)\}ds\\
        &=R(x)J\bu'(x)+\{R(x)P(x)+K(x,x)J\}\bu(x)\\
        &\ -K(x,0)J\bu(0)+\int_0^x\{K(x,s)P(s)
                -\partial_sK(x,s)J\}\bu(s)ds.
\end{aligned}
\end{equation}

If~$R$ is given by~\eqref{eq:Dir.R} and~$K$
satisfies~\eqref{eq:K_b.c_2}, then $JR=RJ$ and $K(x,0)J\bu(0)=0$ for $\bu\in\cD_0$; using now~\eqref{eq:Dir.K} and \eqref{eq:K_b.c_1} in equations \eqref{eq:XD} and \eqref{eq:DX}, we arrive at the equality
\[
    \ell(Q)\sX\bu=\sX\ell(P)\bu
\]
on the domain~$\cD_0$, thus establishing the sufficiency part.

To prove necessity, we equate the integrands and the coefficients of~$\bu$ and $\bu'$ in~\eqref{eq:XD} and \eqref{eq:DX}. The coefficients of~$\bu'$ yield the relation~$JR=RJ$ and so
\[
    R(x)=\left(%
    \begin{array}{cc}
    a(x) & b(x) \\
    -b(x) & a(x) \\
    \end{array}%
    \right).
\]
Now equate the coefficients of~$\bu$ to get the relation
\begin{equation}
\label{eq:PQK}
    JR'(x)+Q(x)R(x)-R(x)P(x)=K(x,x)J-JK(x,x)
\end{equation}
coinciding with~\eqref{eq:K_b.c_1}.
Using the fact that~$KJ-JK$ and~$JKJ+K$ have zero traces, we derive
from~\eqref{eq:PQK} the following system for~$a$ and~$b$:
\begin{equation}
    \label{eq:system_a_b}
    2\frac{d}{dx}\left(%
    \begin{array}{c}
      a \\
      b \\
    \end{array}%
    \right)=\left(%
    \begin{array}{cc}
      \tr J(Q-P) & -\tr (Q-P) \\
      \tr (Q-P) & \tr J(Q-P) \\
    \end{array}%
    \right)\left(%
    \begin{array}{c}
      a \\
      b \\
    \end{array}%
    \right).
\end{equation}
Multiplying the first row by~$a$ and the second by~$b$ and adding
gives
\[
    (a^2+b^2)'=(a^2+b^2)\tr J(Q-P),
\]
so that~$a^2+b^2=c\exp(2\theta_1)$ with $\theta_1$ of~\eqref{eq:Dir.theta} and some constant~$c$. Recalling the assumption $R(0)=I$, we conclude that $c=1$.

Now, upon substituting~$a=\exp(\theta_1)\cos\eta$
and~$b=\exp(\theta_1)\sin\eta$ in the system~\eqref{eq:system_a_b}, we find that~$\eta=\theta_2+c_1$ with a constant~$c_1$. Using again the normalization~$R(0)=I$, we obtain~$\eta=\theta_2+2\pi n, n\in \bZ$. Thus the matrix~$R$ is indeed given by~\eqref{eq:Dir.R}.

Next, the equation for the kernel~$K$ follows from the equality of
the integrands in~\eqref{eq:XD} and~\eqref{eq:DX}. Finally, as the
term~$K(x,0)J\bu(0)$ must vanish for all $\bu\in\cD_0$,
relation~\eqref{eq:K_b.c_2} follows. The proof is complete.
\end{proof}

%Suppose~$P,Q\in L_2((0,1),\mathcal{M}_2)$. Let~$\sX(P,Q)=\cR+\cK$ denote
%transformation operator for~$\cD(P)$ and~$\cD(Q)$. We shall prove
%some theorems concerning this operator.

The above theorem reduces the question on existence of the transformation operator to that on solvability of the system~\eqref{eq:Dir.K}--\eqref{eq:K_b.c_2}. Existence of mild solutions to that system is discussed in the next theorem.

\begin{theorem}
\label{thm:K.exist}
Assume that matrix-valued functions $P$ and $Q$ are in~$L_2((0,1),\mathcal{M}_2)$. Then the system~\eqref{eq:Dir.K}--\eqref{eq:K_b.c_2} has a unique solution in the sense of distributions; moreover, this solution belongs
to~$L_2(\Omega,\mathcal{M}_2)$.
\end{theorem}

\begin{proof}
The proof of this theorem is rather standard and uses reduction to
the equivalent system of integral equations. We only sketch the
main idea and refer the reader to the paper~\cite{Yam:1988} where
the missing details can be found.

Introduce a four-component vector-function~$L=(L_1,L_2,L_3,L_4)^{\mathrm{t}}$ via
%\begin{equation}
\begin{align*}
    L_1&=K_{21}-K_{12}, \quad L_2=K_{22}-K_{11},\\
    L_3&=K_{11}+K_{22}, \quad L_4=K_{12}+K_{21}.
\end{align*}
%\end{equation}
In terms of this vector the system~\eqref{eq:Dir.K} takes the form
\begin{equation} \label{system_L}
    (\partial_x+E\partial_y)L(x,y)=F(x,y)L(x,y),
\end{equation}
where~$E=\operatorname{diag}(1,-1,1,-1)$ and~$F(x,y)$ is a $4\times 4$ matrix-function whose entries are linear combinations of the entries of ~$P(y)$ and~$Q(x)$. The boundary conditions~\eqref{eq:K_b.c_1} and \eqref{eq:K_b.c_2} for~$K$ translate into the relations
\begin{equation}
\label{L_b.c.}
\begin{aligned}
L_k(x,x)&=g_k(x), \quad k=2,4,\\
L_1(x,0)&=L_4(x,0)\\
L_3(x,0)&=-L_2(x,0),
\end{aligned}
\end{equation}
where~$g_2$ and~$g_4$ are respectively the $(1,2)$- and $(1,1)$-entries
of the matrix~$R(x)P(x)-Q(x)R(x)-JR'(x)$. Under the assumptions of the theorem, $g_2$ and $g_4$ belong to $L_2(0,1)$.

We next denote by~$F_i$, $i=\overline{1,4}$, the~$i$-th row of the
matrix~$F$ and rewrite the system~\eqref{system_L}--\eqref{L_b.c.} as a system of the integral equations
\begin{align*}
L_i     &=\int_y^{\frac{x+y}{2}}\!F_i(-s+x+y,s)L(-s+x+y,s)ds
            +g_i\Bigl(\frac{x+y}{2}\Bigr), \; i=2 ,4\\
L_1     &=\int_0^y\!F_1(s+x-y,s)L(s+x-y,s)ds\\
        &\quad +\int_0^{\frac{x-y}{2}}\!F_4(-s+x-y,s)L(-s+x-y,s)ds
            + g_4\left(\frac{x-y}{2}\right)\\
L_3     &=\int_0^y\!F_3(s+x-y,s)L(s+x-y,s)ds\\
        &\quad -\int_0^{\frac{x-y}{2}}\!F_2(-s+x-y,s)L(-s+x-y,s)ds
            - g_2\left(\frac{x-y}{2}\right).
\end{align*}
Applying to the latter the successive approximation
method and using the fact that~$g_2$ and $g_4$ belong to
$L_2(0,1)$, one proves existence of a unique solution~$L$
belonging to $L_2\bigl(\Omega,{\mathbb C}^4\bigr)$. This gives a unique
mild solution~$K$ of the original hyperbolic
system~\eqref{eq:Dir.K}--\eqref{eq:K_b.c_2} and proves that~$K\in
L_2(\Omega,\mathcal{M}_2)$.
\end{proof}

Throughout the rest of the paper, we shall assume that the
matrix-valued potentials $P$ and $Q$ are Hermitian, i.e., that
$P^*(x) = P(x)$ and $Q^*(x) = Q(x)$ a.e.\ on $[0,1]$. Then the
corresponding Dirac operators~$\sD(P)$ and $\sD(Q)$ are
self-adjoint and have simple discrete spectra. Moreover, the
eigenvalues of $\mu_n(P)$ of $\sD(P)$ can be labeled by $n\in\bZ$
so that $\mu_n = \pi n + \tfrac12\tr P + \mathrm{o(1)}$ as
$|n|\to\infty$ (see~\cite{LevSar:1991}), and similarly for the eigenvalues
of~$\sD(Q)$.

Just as for an operator pencil $T_{p,q}$, we define the
\emph{spectral data} for a Dirac operator~$\sD$ as the
set~$\{(\la,\al)\mid \la \in \sigma(\sD)\}$ of all eigenpairs
$(\la,\al)$ composed of eigenvalues $\la$ and the corresponding
norming constants~$\al$. We denote by~$\sd(P)$ (resp. by~$\sd(Q)$)
the spectral data of~$\sD(P)$ (resp., the spectral data
of~$\sD(Q)$).

Also, $\sX=\sX(P,Q)$ will stand for the transformation operator of the form~\eqref{eq:Tr.op} for the differential expressions~$\ell(P)$ and $\ell(Q)$ on the domain~$\cD_0$, with $R$ given by~\eqref{eq:Dir.R} and \eqref{eq:Dir.theta}. We shall write $\sX=\sR + \sK$, where $\sR\bu(x):=R(x)\bu(x)$ is the operator of multiplication by $R$ and
\[
    \sK\bu(x):= \int_0^x K(x,s)\bu(s)\,ds
\]
is the corresponding integral operator. We observe that for Hermitian~$P$ and $Q$ the functions~$i \theta_1$ and $\theta_2$ are real valued; in particular, the operator~$\sR$ is unitary.

Before discussing further properties of the transformation operator~$\sX$, the following simple but useful remark seems in place.

\begin{remark}\label{rem:Dir.X}
Since the transformation operator~$\sX$ intertwines the Dirac differential expressions~$\ell(P)$ and $\ell(Q)$, it is straightforward to see that the relation
\[
    \bigl(\ell(P)-\lambda\bigr)\bu = \mathbf{f}
\]
holds if and only if for $\bv:=\sX\bu$ and $\bg:=\sX\mathbf{f}$
one gets
\[
    \bigl(\ell(Q) - \lambda\bigr)\bv = \bg.
\]
\end{remark}

\begin{lemma}\label{lem:K=0}
Assume that $P$ and~$Q$ are Hermitian and that~$\sd(P)=\sd(Q)$.
Then the integral operator~$\sK$ in the transformation operator~$\sX=\sX(P,Q)$ is the zero operator.
\end{lemma}

\begin{proof}
We enumerate the eigenpairs in $\sd(P)=\sd(Q)$ as $(\la_n,\al_n)$ for $n\in\bZ$ and denote by $\bu_n$ the eigenfunction of the operator~$\sD(P)$ corresponding to the eigenvalue~$\lambda_n$ and normalized via $\bu_n(0)=(1,0)^\mathrm{t}$. Then
$\bv_n:=\sX(P,Q)\bu_n$ is the corresponding eigenfunction for the operator~$\sD(Q)$ satisfying the same initial condition.

The sequences $(\bu_n)_{n\in\bZ}$ and $(\bv_n)_{n\in\bZ}$ form orthogonal bases of the Hilbert space~$L_2\bigl((0,1),\bC^2\bigr)$; moreover, $\|\bu_n\|=\|\bv_n\|=\sqrt{\alpha_n}$. Thus we conclude that the operator~$\sX$ is unitary, i.e., that
\[
      (\sR+\sK)^\ast(\sR+\sK)=\sI,
\]
where~$\sI$ is the identity operator in~$L_2\bigl((0,1),\bC^2\bigr)$. Since~$\sR$ is unitary, the last equality may be rewritten as
\[
    (\sI+\sR^{-1}\sK)^{\ast}(\sI+\sR^{-1}\sK)=\sI.
\]

We recall that~$\sK$ (and thus~$\sR^{-1}\sK$) is an integral
operators with lower-triangular kernel belonging to
$L_2(\Omega,\mathcal{M}_2)$. Slightly modifying the arguments
of~\cite[Sec.~IV.1]{RieNag:1990}, one sees that $\sR^{-1}\sK$ is a
Volterra operator, whence the inverse $(\sI+\sR^{-1}\sK)^{-1}$
exists and is given by the Neumann series,
\[
    (\sI+\sR^{-1}\sK)^{-1}
        = \sI - \sR^{-1}\sK + (\sR^{-1}\sK)^2 + \cdots
        = \sI + \widetilde{\sK}
\]
with $\widetilde{\sK}$ being an integral operator with lower-triangular kernel. On the other hand, the operator~$(\sR^{-1}\sK)^\ast$ is an integral operator with upper-triangular kernel, and the relations
\[
    \sI+(\sR^{-1}\sK)^{\ast} = (\sI+\sR^{-1}\sK)^{-1}  = \sI + \widetilde{\sK}
\]
imply that $(\sR^{-1}\sK)^{\ast}=\widetilde{\sK}=0$. Therefore $\sK=0$, and the proof is complete.
\end{proof}

The following theorem gives necessary and sufficient conditions on transformation operators $\sX(P,Q)$ in order that the Dirac operators $\sD(P)$ and $\sD(Q)$ should have the same spectral data. It will essentially be used in the reconstruction procedure of the next section.

\begin{theorem}
\label{thm:sp.data.coinc.crit}
Assume that the matrix potentials~$P$ and~$Q$ are Hermitian. Then the spectral data for the operators~$\sD(P)$ and~$\sD(Q)$ coincide, i.e. $\sd(P)=\sd(Q)$,
if and only if the transformation operator~$\sX(P,Q)$ for $\ell(P)$ and $\ell(Q)$ on the domain~$\cD_0$ only contains the unitary part $\sR$ (i.e., $\sK=0$) and~$\theta_2(1)=\pi n$ for some $n\in\bZ$.
\end{theorem}

\begin{proof}
\emph{Necessity}. If the spectral data for the operators~$\sD(P)$
and~$\sD(Q)$ coincide, then $\sK=0$ by Lemma~\ref{lem:K=0} and thus $\sX(P,Q)=\sR$.
Take an arbitrary eigenvalue~$\lambda$ of $\sD(P)$ and denote by $\bu=(u_1,u_2)^{\mathrm{t}}$ the corresponding eigenfunction. Then
\[
    \bv(x):=\sR\bu
        =\exp\bigl(\theta_1(x)\bigr)
            \left(\begin{array}{c}
                 \cos\theta_2(x)u_1(x)+\sin\theta_2(x)u_2(x)\\
                -\sin\theta_2(x)u_1(x)+\cos\theta_2(x)u_2(x)
                  \end{array}%
            \right)
\]
is the eigenfunction for the operator~$\sD(Q)$ corresponding to the same eigenvalue~$\lambda$.
Observe now that the second components of~$\bu(1)$ and~$\bv(1)$ are equal
to zero. Since
\[
    \bv(1)=\exp(\theta_1(1))\left(%
\begin{array}{c}
   \cos\theta_2(1)u_1(1)\\
   -\sin\theta_2(1)u_1(1)
\end{array}%
\right)
\]
and~$u_1(1)\ne0$, we conclude that $\sin\theta_2(1)=0$, and thus that $\theta_2(1)=\pi n$ for an integer~$n$ as claimed. This completes the proof of the necessity parts.

\smallskip

\emph{Sufficiency}. Suppose that $\sX(P,Q)=\sR$ and that $\lambda$ is an eigenvalue of the operator~$\sD(P)$ with a corresponding eigenfunction~$\bu$.
Consider the vector~$\bv=\sR \bu$; then $\ell(Q)\bv = \lambda\bv$. The assumption~$\theta_2(1)=\pi n$ for an integer~$n$ implies that $R(1)$ is a multiple of the identity matrix~$I$. Since also $R(0)=I$, we conclude that the second component~$v_2$ of $\bv$ vanishes at both endpoints and thus $\bv\in\dom\sD(Q)$.

The above implies that the spectrum of $\sD(P)$ is contained in
the spectrum of~$\sD(Q)$. As $\sR$ is boundedly invertible, the
roles of $P$ and $Q$ can be interchanged, thus showing that the
spectra of the operators coincide. Finally, since the operator
$\sR$ is unitary and preserves the initial conditions, the norming
constants for $\sD(P)$ and $\sD(Q)$ are equal, i.e.,
$\sd(P)=\sd(Q)$. The proof is complete.
\end{proof}

%%%%%%%%%%%%%%%%%%%%%%%%%%%%%%%%%%%%%%%%%%%%%%%%%%%%%

\section{Reconstruction of the pencil: Existence}\label{sec:Dir-rec}

%%%%%%%%%%%%%%%%%%%%%%%%%%%%%%%%%%%%%%%%%%%%%%%%%%%%%

Our aim in this section is to prove Theorem~\ref{thm:pre.exist} on existence of a pencil~$T_{p,q}$ with potentials~$p\in L_{2,\bR}(0,1)$ and $q\in W_{2,\bR}^{-1}(0,1)$ having the prescribed element~$\sd$ of~$\SD$ as its spectral data. The uniqueness of reconstruction will be dealt with in the next section.

\begin{proof}[Outline of the proof of Theorem~\ref{thm:pre.exist}]
First, we construct a Dirac operator~$\sD(Q)$ in the ``shifted'' AKNS form whose nonzero spectrum and corresponding norming constants are given by~$\sd$.
Then we use the transformation operator technique to transform $\sD(Q)$ to another Dirac operator~$\sD(P)$ with the same spectral data and with potential~$P=(p_{ij})$ of the form~\eqref{eq:Dir.P}. Setting then
\begin{equation}
    \label{eq:pq}
        p := \frac{p_{22}}{2}, \quad
        q := -p_{12}'+p_{12}^2
\end{equation}
and recalling the results of Section~\ref{sec:Dir}, we conclude that the operator pencil~$T_{p,q}$ is a solution of the inverse spectral problem, thus completing the proof.
\end{proof}

The rest of this section contains details of constructing the potential $Q$ of the ``shifted'' AKNS form and transforming it to a $P$ of the form~\eqref{eq:Dir.P}.

We start with taking an arbitrary set~$\sd=\{(\la_n,\al_n)\}_{n\in\bZ^*}$ in~$\SD$. The enumeration of~$\la_n$ is uniquely determined by the requirement that $\la_{-1}<0$ and $\la_1>0$ and fixes the number $h$ in the asymptotic representation of part~(i) of Definition~\ref{def:SD}. We then take an arbitrary $\alpha_0>0$, put $\la_0:=0$, and augment the set~$\sd$ with~$(\la_0,\al_0)$ to become $\sd^*$.

Now we recall the following facts from the inverse spectral theory
for AKNS Dirac operators,
see~\cite{LevSar:1991,AlbHryMk:2005:RJMP}. Denote by $\cQ_0$ the
set of $2\times2$ matrix-valued functions $Q_0$ of the AKNS normal
form, namely
\begin{equation}\label{eq:ex.Q0}
    \cQ_0:=\left\{ Q_0=
    \begin{pmatrix}
        q_1 & q_2 \\ q_2 & -q_1
    \end{pmatrix} \mid q_j \in L_{2,\bR}(0,1)\right\}.
\end{equation}
For $Q_0\in \cQ_0$, the Dirac operator~$\sD(Q_0)$ is  self-adjoint, has a simple discrete spectrum, and its eigenvalues can be enumerated as $\la_n(Q_0)$, $n\in\bZ$, so that $\la_n(Q_0)$ increase with $n$ and $\la_n(Q_0) = \pi n + \tilde \la_n(Q_0)$, with an $\ell_2(\bZ)$-sequence~$(\tilde \la_n)$. The corresponding norming constants $\al_n(Q_0)$ are positive and the remainders $\tilde \alpha_n(Q_0):=\alpha_n(Q_0)-1$ form an $\ell_2(\bZ)$-sequence.

Conversely, it follows from the results
of~\cite{AlbHryMk:2005:RJMP,Ser:2006} that every set
$\{(\la_n,\alpha_n)\}_{n\in\bZ}$ with $\la_n$ and $\al_n$
possessing the above properties is the set of spectral data for a
unique AKNS Dirac operator $\sD(Q_0)$ with $Q_0\in\cQ_0$.

The set $\sd^*$ (i.e., the set $\sd$ augmented with the pair $(\la_0,\al_0)$ as above) looks like spectral data for an AKNS Dirac operator, save that the asymptotics of~$\la_n$ is shifted by some number~$h$; cf.~the definition of the set $\SD$. For $h\in\bR$, we denote by
\[
    \cQ_h:= \{Q_0+hI \mid Q_0 \in \cQ_0\}
\]
the set of $h$-\emph{shifted} AKNS matrix potentials; then the following holds true.

\begin{proposition}\label{pro:exist.Q}
For an arbitrary $\sd \in\SD$, fix $h\in\bR$ in the representation of~(i) in Definition~\ref{def:SD} so that $\la_{-1}<0$ and $\la_1>0$, and denote by $\sd^*$ augmentation of $\sd$ by a pair $(\la_0,\al_0)$, with $\la_0:=0$ and a positive $\al_0$. Then there exists a unique potential $Q\in\cQ_h$ such that $\sd^*$ gives the spectral data for the Dirac operator~$\sD(Q)$.
\end{proposition}

For an arbitrary $Q\in L_2\bigl((0,1),\mathcal{M}_2\bigr)$, we introduce the set
\[
    \Iso(Q):=\{\tilde Q\in L_2\bigl((0,1),\mathcal{M}_2\bigr) \mid
    \sd(\tilde Q)=\sd(Q)\}
\]
of potentials \emph{isospectral} with a given~$Q$ and denote by~$\cP$ the set of all potentials of the form~\eqref{eq:Dir.P}, i.e.,
\[
    \mathcal{P}:=\{P=(p_{ij})_{i,j=1}^2 \mid
        p_{ij}\in L_{2,\mathbb{R}}(0,1),
        \ p_{11}=0,\ p_{12}=p_{21}\}.
\]
The following result is crucial in constructing a pencil~$T_{p,q}$ with the given spectral data.

\begin{theorem}\label{thm:exist_uniq_P}
Assume that~$Q \in \cQ_h$ is such that $\la=0$ is an eigenvalue of the Dirac operator~$\sD(Q)$. Then there exists a unique~$P\in\mathcal{P}$ such that
\(
    \Iso(Q)\cap\mathcal{P}=\{P\}.
\)
\end{theorem}

\begin{proof}
We divide the proof into three steps. First we prove that there is a unique $P\in\cP$ with the property that the transformation operator~$\sX(P,Q)$ between $\ell(P)$ and $\ell(Q)$ on~$\cD_0$ is just the operator~$\sR$ of multiplication by the matrix-valued function $R$ of~\eqref{eq:Dir.R}--\eqref{eq:Dir.theta}. On the second step we show that this $P$ is indeed isospectral with $Q$. Finally, we explain why there is no other potentials in $\Iso(Q)\cap\mathcal{P}$.

\underline{Step~1.} We write $Q \in \cQ_h$ as
\[
    Q = hI + \begin{pmatrix}q_1 & q_2 \\ q_2 & -q_1\end{pmatrix}
\]
with real-valued $q_1$ and $q_2$ in $L_2(0,1)$. Let~$\theta$ be an absolutely continuous function and $R=e^{\theta J}$; then $R$ commutes with $J$ and one finds that
\[
    R^{-1}\Bigl(J\frac{d}{dx} + Q\Bigr)R = J\frac{d}{dx} + R^{-1}JR' + R^{-1}QR.
\]
Therefore
\(
    \sR \ell(P) = \ell(Q) \sR
\)
for a unique matrix potential~$P$ equal to
\begin{equation}\label{eq:P:theta}
    \begin{aligned}
     P  & = R^{-1}JR'+R^{-1}QR \\
        & = (h-\theta')I
          + \left(
            \begin{array}{cc}
                q_{1}\cos2\theta - q_{2}\sin 2\theta
              & q_{1}\sin2\theta + q_{2}\cos 2\theta \\
                q_{1}\sin 2\theta + q_{2}\cos2\theta
              &-q_{1}\cos2\theta + q_{2}\sin 2\theta
            \end{array}%
            \right).
    \end{aligned}
\end{equation}
The potential~$P$ defined by~\eqref{eq:P:theta} belongs to the
set~$\cP$ if and only if
\begin{equation}\label{eq:theta}
     -\theta'+q_{1}\cos2\theta-q_{2}\sin 2\theta+h=0.
\end{equation}
There exists a unique solution of this equation over~$[0,1]$ satisfying the initial condition~$\theta(0)=0$. We call this solution $\theta_2$ and determine the corresponding potential~$P\in\cP$ through~\eqref{eq:P:theta} with $\theta=\theta_2$.  A straightforward calculation shows that this $\theta_2$ and $\theta_1\equiv0$ verify the relations~\eqref{eq:Dir.theta}; in particular, $\theta_2' =  h-\tfrac12p_{22}$. By Theorem~\ref{thm:Dir.tr-op}, the operator $\sR$ of multiplication by the matrix-valued function~$R=e^{\theta_2 J}$ of~\eqref{eq:Dir.R} is indeed the transformation operator between $\ell(P)$ and $\ell(Q)$ on~$\cD_0$.

\underline{Step~2.}
Next we claim that $\theta_2(1)=\pi n$ for some $n\in\bZ$.
By construction, $\sR^{-1}\ell(Q)\sR=\ell(P)$, so that the operator $\widetilde{\sD}(P):=\sR^{-1}\sD(Q)\sR$ is a Dirac operator defined via $\widetilde{\sD}(P)\bu = \ell(P)\bu$ on the domain consisting of those $\bu=(u_1,u_2)^\mathrm{t}\in L_2\bigl((0,1),\bC^2\bigr)$ for which $\sR\bu\in\dom\sD(Q)$.
The assumption that $\lambda=0$ is an eigenvalue of the Dirac operator~$\sD(Q)$ and the similarity of $\sD(Q)$ and $\widetilde \sD(P)$ imply that $\lambda=0$ is in the spectrum of~$\widetilde \sD(P)$, and we denote by $\bu^0=(u_1,u_2)^\mathrm{t}$ the corresponding eigenfunction.
As $R(0)=I$ and $\bv^0=(v_1,v_2)^\mathrm{t}:=\sR\bu^0$ is in the null-space of~$\sD(Q)$, we see that $u_2(0)=v_2(0)=0$, and thus~$u_2\equiv0$ by~\eqref{eq:Dir.system1}. Therefore,
\[
    \bv^0(1)
        = R(1)\bu^0(1)
        = \binom{u_1(1)\cos\theta_2(1)}{-u_1(1)\sin\theta_2(1)}.
\]
As~$u_1(1)\ne0$ and $v_2(1)=0$, we conclude that~$\sin\theta_2(1)=0$ and thus that $\theta_2(1)=\pi n$ for some~$n\in\bZ$.

Henceforth we have constructed a potential~$P\in\cP$ and a unitary operator~$\sR$ of the form~\eqref{eq:Dir.R}--\eqref{eq:Dir.theta} with $\theta_2(1)=\pi n$ for some $n\in\bZ$ so that $\sR$ is the transformation between~$\ell(P)$ and $\ell(Q)$ on~$\cD_0$. By Theorem~\ref{thm:sp.data.coinc.crit}, $P$ then belongs to~$\Iso(Q)$.

\underline{Step~3.}
Finally, suppose there is another~$P_1$ in the set~$\Iso(Q)\cap\cP$. By Theorem~\ref{thm:sp.data.coinc.crit}, the transformation operator $\sX(P_1,Q)$ between~$\ell(P_1)$ and~$\ell(Q)$ is equal to the operator~$\sR_1$ of multiplication by a matrix-valued function~$e^{\vartheta_1(x)I+\vartheta_2(x)J}$, where $\vartheta_1$ and $\vartheta_2$ are given as in~\eqref{eq:Dir.theta}, but with $P$ replaced by~$P_1$. Since $\vartheta_1\equiv0$ for Hermitian $P_1$ and $Q$ with real entries, we conclude that $P_1=P$ by step~1, thus completing the proof of Theorem~\ref{thm:exist_uniq_P}.
\end{proof}

Theorem~\ref{thm:exist_uniq_P} implies that there is a unique $P\in\cP$ such that $\sd(P) = \sd^*$, for every augmentation $\sd^*$ of the set $\sd$. As explained at the beginning of this section, the Dirac operator~$\sD(P)$ with every such $P$ yields a pencil~$T_{p,q}$ with spectral data~$\sd$, thus establishing Theorem~\ref{thm:pre.exist}.

However, as the augmented set~$\sd^*$ depends on the arbitrary choice of~$\alpha_0>0$,  different $\al_0$ lead to different~$P\in\cP$, and, plausibly, to different quadratic operator pencils~$T_{p,q}$. This uniqueness issue is discussed in detail in the next section.

%%%%%%%%%%%%%%%%%%%%%%%%%%%%%%%%%%%%%%%%%%%%%%%%%%%%%%%%%%%%%%%%%%

\section{Reconstruction of the pencil: Uniqueness}\label{sec:Dir-rec-uniq}

%%%%%%%%%%%%%%%%%%%%%%%%%%%%%%%%%%%%%%%%%%%%%%%%%%%%%%%%%%%%%%%%%%

In this section, we prove Theorem~\ref{thm:pre.uniq} and thus complete the study of the inverse spectral problem for quadratic pencils~$T_{p,q}$. Namely, we show that the matrix potentials $P\in\cP$ constructed in the previous section lead to the same~$p$ and $q$ for all choices of the positive parameter~$\al_0$.

We again start with an arbitrary element $\sd \in \SD$, set $\la_0:=0$, choose an arbitrary positive number~$\al_0$, and augment~$\sd$ with the pair~$(\la_0,\al_0)$. Then we construct a shifted AKNS potential $Q\in\cQ_h$ and the corresponding potential~$P\in\Iso(Q)\cap\cP$ whose spectral data coincide with the augmented set, see Proposition~\ref{pro:exist.Q}  and Theorem~\ref{thm:exist_uniq_P}.

Choose now a positive number $\tilde \al_0$ different from~$\al_0$.
Augmenting the set~$\sd$ with~$(\la_0,\tilde{\al}_0)$, we obtain
spectral data for another Dirac operator $\sD(\tilde Q)$ with potential~$\tilde Q\in\cQ_h$, and construct the corresponding potential~$\tilde P\in\Iso(\tilde Q)\cap \cP$.

It turns out that the potentials~$Q$ and~$\tilde{Q}$ are related
via the so-called double commutation transformations, see details
in~\cite[Sec.~3]{Tes:1998} and~\cite{AlbHryMyk:2007:JDE}.

\begin{proposition}\label{pro:uniq.Q-and-tildeQ}
Suppose that two potentials~$Q$ and~$\tilde{Q}$ from~$\cQ_h$ are
as described above, i.e., that the spectra of the corresponding Dirac
operators $\sD(Q)$ and $\sD(\tilde Q)$ coincide and the norming constants $\al_n$ and $\tilde \al_n$ only differ for $n=0$. Then
\begin{equation}
    \label{Q}
    \widetilde{Q}=Q+Q_*,
\end{equation}
where
\begin{align}\label{eq:uniq.Q*}
    Q_* &= c(x,\alpha_*)[\mathbf{v}(x)\mathbf{v}^{\mathrm{t}}(x)J
                         -J\mathbf{v}(x)\mathbf{v}^{\mathrm{t}}(x)],\\
    c(x,\alpha_*)
        &:=-\frac{\alpha_*}{1 + \alpha_*\int_0^x
                     \mathbf{v}^{\mathrm{t}}(s)\mathbf{v}(s)ds},\quad
    \alpha_*:=\frac{1}{\tilde{\alpha}_0}-\frac{1}{\alpha_0}
    \label{eq:uniq.c-alpha*}
\end{align}
and~$\mathbf{v}$ is an eigenfunction of the operator~$\sD(Q)$
corresponding to the eigenvalue~$\la_0=0$.
\end{proposition}

Recalling the way the operators~$\sD(Q)$ and $\sD(P)$ are related and Remark~\ref{rem:Dir.0} on the form of the eigenvector $\bu$ of~$\sD(P)$ corresponding to the eigenvalue~$\la_0=0$, we can write a more explicit formula for~$Q_*$.

Indeed, by Theorem~\ref{thm:sp.data.coinc.crit}
the transformation operator~$\sX(P,Q)$ for the Dirac differential expressions~$\ell(P)$ and~$\ell(Q)$ on the set~$\cD_0$ is just the operator~$\sR$ of multiplication by the matrix-valued function~$R=e^{\theta_2 J}$, with $\theta_2$ being the solution of~\eqref{eq:theta} satisfying $\theta(0)=0$. Therefore
 $\bv = \sR \bu = e^{\theta_2 J} \bu$;
as $\bu=(u_1,0)^{\mathrm{t}}$ and $(e^{\theta_2 J})^{\mathrm{t}}=e^{-\theta_2 J}$,
we easily compute that $\bv^{\mathrm{t}} \bv = \bu^{\mathrm{t}} \bu = u_1^2$ and
\[
    \bu \bu^{\mathrm{t}}J - J \bu \bu^{\mathrm{t}} = u_1^2 J_1,
\]
with
\[
     J_1 := \left(%
             \begin{array}{cc}
                 0 & 1 \\
                 1 & 0 \\
             \end{array}%
           \right).
\]
Observing that the matrices $J$ and $J_1$ anticommute, we conclude that
 $e^{\theta_2 J}J_1=J_1e^{-\theta_2 J}$; henceforth,
\begin{align*}
    \bv \bv^{\mathrm{t}}J - J \bv \bv^{\mathrm{t}}
        &= e^{\theta_2 J}[\bu \bu^{\mathrm{t}}J - J \bu \bu^{\mathrm{t}}] e^{-\theta_2 J}\\
        &= u_1^2 e^{\theta_2 J}J_1 e^{-\theta_2 J} = u_1^2 e^{2\theta_2 J}J_1.
\end{align*}
Set $w(x):= 1 + \alpha_*\int_0^x u_1^2(s)\,ds$; then
     $c(x,\alpha_*)u_1^2(x) = - w'(x)/w(x) = -[\log w(x)]'$,
thus resulting in the following form of the potential $Q_*$.

\begin{corollary}\label{cor:uniq.Q-and-tildeQ}
For the potentials~$Q$ and~$\tilde{Q}$ of Proposition~\ref{pro:uniq.Q-and-tildeQ}, relation~\eqref{Q} holds with
\begin{equation}
\label{eq:uniq.Qstar}
\begin{aligned}
    Q_*(x)
        & = -[\log w(x)]' e^{2\theta_2(x) J}J_1 \\
        & = -[\log w(x)]'
        \left(%
        \begin{array}{cc}
            \sin 2\theta_2(x) & \cos 2\theta_2(x) \\
            \cos 2\theta_2(x) & -\sin 2\theta_2(x) \\
        \end{array}%
       \right),
\end{aligned}
\end{equation}
where~$\theta_2$ is the solution of~\eqref{eq:theta} satisfying $\theta(0)=0$, $u_1$ is the first component of the eigenvector~$\bu$ of $\sD(P)$ corresponding to the eigenvalue~$\la_0=0$, and
\begin{equation}\label{eq:uniq.d}
    w(x) = 1 + \alpha_*\int_0^x u_1^2(s)\,ds.
\end{equation}
\end{corollary}

Now we use the explicit formulae~ \eqref{eq:P:theta} and \eqref{eq:theta} determining the potential~$P$ from~$Q$ and the analogous formula for $\tilde P$ and $\tilde Q$ to derive the crucial result relating $P$ and $\tilde P$.

\begin{lemma}\label{lem:uniq.P-and-Ptilde}
For the entries $p_{ij}$ and $\tilde p_{ij}$ of the matrices $P$ and $\tilde P$ constructed above, the following relations hold:
\[
    \tilde p_{22} =  p_{22}, \qquad
    \tilde p_{12} =  p_{12} - (\log w)',
\]
with the function $w$ of~\eqref{eq:uniq.d}.
\end{lemma}

\begin{proof}
First we recall that $\theta_2$ is the unique solution of equation~\eqref{eq:theta},
\[
    -\theta'+q_{1}\cos2\theta-q_{2}\sin 2\theta+h=0,
\]
satisfying the initial condition $\theta(0)=0$;
here $q_1$ and $q_2$ are the entries of the AKNS part $Q_0$ of the potential~$Q$ as in~\eqref{eq:ex.Q0}. Likewise, $\tilde \theta_2$ is the unique solution of
\[
    -\tilde\theta'+\tilde q_{1}\cos2\tilde\theta-\tilde q_{2}\sin 2\tilde\theta+h=0
\]
satisfying $\tilde\theta(0)=0$, with $\tilde q_1$ and $\tilde q_2$ having similar meaning.
Equality~\eqref{eq:uniq.Qstar} together with Proposition~\ref{pro:uniq.Q-and-tildeQ}
allow to recast the latter equation for $\tilde\theta_2$ as
\[
     -\tilde\theta'+q_{1}\cos2\tilde\theta-q_{2}\sin 2\tilde\theta+h
            +(\log w)'\sin(2\tilde\theta - 2\theta_2) =0.
\]
Observe that $\tilde\theta\equiv\theta_2$ is a solution of this equation satisfying the initial condition~$\tilde\theta(0)=0$; therefore, uniqueness of solutions yields
$\tilde\theta_2 \equiv \theta_2$.

The potential $\tilde P$ is related to $\tilde Q = Q + Q_*$ through a formula analogous to~\eqref{eq:P:theta}, i.e.,
\[
    \tilde P = \tilde R^{-1}J\tilde R'+\tilde R^{-1}(Q + Q_*)\tilde R,
\]
with $\tilde R = e^{\tilde\theta_2 J } = e^{\theta_2 J}=R$. Therefore we find that
\begin{align*}
    \tilde P - P
        &= R^{-1}Q_* R
         = e^{-\theta_2 J} [-(\log w)' e^{\theta_2 J}J_1e^{-\theta_2 J}]e^{\theta_2 J}\\
        &= - (\log w)' J_1.
\end{align*}
As a result, $\tilde p_{22}=p_{22}$ and $\tilde p_{12} = p_{12} -(\log w)'$, and the lemma is proved.
\end{proof}

\begin{corollary}\label{cor:uni.P-and-tildeP}
The potentials~$P$ and~$\tilde{P}$ constructed above generate the same operator pencil~$T_{p,q}$.
\end{corollary}

\begin{proof}
In view of~\eqref{eq:pq} and the above lemma it remains to show that
\[  %\begin{equation}\label{ee2}
    -\tilde{p}_{12}'+\tilde{p}_{12}^2 = -p_{12}'+p_{12}^2.
\]  %\end{equation}
By Remark~\ref{rem:Dir.0} we have $u_1(x)=\exp\{\int_0^x p_{12}(s)ds\}$, so that $u_1' = p_{12}u_1$.
Since $w'=\alpha_*u_1^2$, $w''=2\alpha_*u_1'u_1 = 2p_{12}w'$, and
\[
    (\log w)'' = \frac{w''}{w} - \Bigl(\frac{w'}{w}\Bigr)^2
            = 2p_{12}(\log w)' - [(\log w)']^2,
\]
upon substituting $\tilde p_{12} = p_{12} -(\log w)'$ we find that
\begin{align*}
    -\tilde p'_{12} + \tilde p_{12}^2
        &= - p'_{12}  + (\log w)'' + p^2_{12}- 2p_{12}(\log w)' + [(\log w)']^2   \\
        &= - p'_{12} + p^2_{12}
\end{align*}
as claimed. The proof is complete.
\end{proof}

\begin{proof}[Proof of Theorem~\ref{thm:pre.uniq}]
Suppose there are two pencils, $T=T_{p,q}$ and $\widehat T=T_{\hat p, \hat q}$, satisfying
assumption~(A) and having the same spectral data in~$\SD$. As explained in Section~\ref{sec:Dir}, these pencils
lead to two Dirac operators $\sD(P)$ and $\sD(\widehat P)$ with some potentials $P$ and $\widehat P$ in~$\cP$.

The spectral data for $\sD(P)$ and $\sD(\widehat P)$ can only differ
at the norming constant for the eigenvalue $\lambda=0$; denote
these norming constants by $\alpha_0$ and $\hat\alpha_0$
respectively.

Now take $\tilde \alpha_0=\hat\alpha_0$ and construct the potential $\tilde P\in\cP$ as explained at the beginning of this section. By Corollary~\ref{cor:uni.P-and-tildeP}, $\tilde P$ and $P$ generate the same pencil~$T$. On the other hand, the potentials $\tilde P$ and $\hat P$
are isospectral and belong to~$\cP$ and thus coincide by Theorem~\ref{thm:exist_uniq_P}. Therefore, $T$ and $\widehat T$ coincide as well, and the proof is complete.
\end{proof}

%%%%%%%%%%%%%%%%%%%%%%%%%%%%%%%%%%%%%%

\section{Reconstruction algorithm and some extensions}\label{sec:alg}

%%%%%%%%%%%%%%%%%%%%%%%%%%%%%%%%%%%%%%%%%%%%%%

The proof of existence theorem (Theorem~\ref{thm:pre.exist}) contains explicit steps forming reconstruction algorithm. Namely, given an arbitrary element~$\sd$ of~$\SD$, we construct a quadratic pencil $T_{p,q}$---i.e., the Sturm--Liouville eigenvalue problem~\eqref{eq:intr.spr} with potentials $p\in L_{2,\mathbb{R}}(0,1)$ and $q \in W_{2,\mathbb{R}}^{-1}(0,1)$---in the following way:
\begin{enumerate}
    \item fix the enumeration $(\la_n,\al_n)$, $n\in\bZ^*$, of the pairs $(\la,\al)$ in $\sd$ so that $\la_n$ increase, $\la_{-1}<0$, $\la_1>0$, and determine the shift $h$ from the asymptotic representation of $\la_n$;
    \item augment the set~$\sd$ with a pair~$(\la_0,\al_0)$, where~$\lambda_0=0$ and~$\alpha_0$ is an  arbitrary positive number;
    \item construct a Dirac operator $\sD(Q)$ with potential~$Q$ in the shifted AKNS class~$\cQ_h$ whose spectral data coincide with the augmented set $\sd^*$ (see Proposition~\ref{pro:exist.Q});
    \item find the corresponding potential~$P\in\mathcal{P}\cap \Iso(Q)$ via~\eqref{eq:P:theta}--\eqref{eq:theta};
    \item compute the potentials~$p$ and~$q$ using formulas~\eqref{eq:pq}.
\end{enumerate}

We finish the paper with several comments. Firstly, the above
algorithm can be used to reconstruct the potentials $p$ and $q$
under different assumptions on their regularity. Namely, if $p$
and a primitive $r$ of~$q$ belong to~$L_s(0,1)$ with $s\ge1$, then
the corresponding set $\SD$ of spectral data allows an explicit
description (the only difference with the $s=2$ case is in the
decay of the remainders $\tilde \la_n$ and $\tilde \al_n$) and the
steps of reconstruction are as above; cf.\ the characterization of
the spectral data for the corresponding class of Dirac operators
in~\cite{AlbHryMk:2005:RJMP}. Similar characterization of the set
$\SD$ is available if $p$ and $r$ belong to~$W_2^{s}(0,1)$ with
$s\ge0$; cf.\ the results of~\cite{HryMyk:2006:PEMS,SavShk:2006} on
eigenvalue asymptotics for Sturm--Liouville operators with
potentials in Sobolev spaces.

Secondly, the approach described is not restricted to the Dirichlet boundary conditions and
can be used to reconstruct energy-dependent Sturm--Liouville equations under quite general separated boundary conditions.

Finally, reconstruction from different sets of spectral data
(e.g., from two spectra, or from Hochstadt--Lieberman mixed data)
using this method is also possible and will be considered
elsewhere, cf.~\cite{Pro:2011c}. One can also get a Hochstadt-type
results~\cite{Hoch:73} on explicit form of the potentials when only finitely many spectral data are changed.

%%%%%%%%%%%%%%%%%%%%%%%%%%%%%%%%%%%%%%%

\bigskip

\noindent\emph{Acknowledgement.} The authors acknowledge support from the Isaac Newton Institute for Mathematical Sciences at the University of Cambridge for participation in the programme \emph{``Inverse Problems''}, during which part of this work was done.

%%%%%%%%%%%%%%%%%%%%%%%%%%%%%%%%%%%%%%%%%%%

\bibliographystyle{abbrv}

\bibliography{inv-lambdaSL}

\end{document}